\documentclass[12pt,leqno]{amsart}
\usepackage[letterpaper,top=2cm,bottom=2cm,left=3cm,right=3cm,marginparwidth=1.75cm]{geometry}
\usepackage[OT1,T2A]{fontenc}

\usepackage{comment}
\usepackage[english]{babel}
\usepackage{amsmath,amssymb,amsthm,amscd}
\usepackage{latexsym}
\usepackage{mathrsfs}
\usepackage{array}
\usepackage[dvipdfmx]{graphicx}
\usepackage{color}
\usepackage{mathtools}
\usepackage{tikz-cd} 
\usepackage{url}
\newcommand{\gal}{\mathrm{Gal}}
\newcommand{\GL}{\mathrm{GL}}

\newcommand{\fr}{\mathrm{Frob}}

\newcommand{\mbZ}{\mathbb{Z}}
\newcommand{\mbQ}{\mathbb{Q}}
\newcommand{\mcA}{\mathscr{A}} 

\definecolor{e-mail}{rgb}{0,.40,.80}
\definecolor{reference}{rgb}{.20,.60,.22}
\definecolor{mrnumber}{rgb}{.80,.40,0}
\definecolor{citation}{rgb}{0,.40,.80}
\usepackage[breaklinks=true,colorlinks=true,linkcolor=reference, citecolor=citation, urlcolor=e-mail]{hyperref}

\theoremstyle{plain}
\newtheorem{theorem}{Theorem}[section]
\newtheorem*{theorem*}{Theorem} 
\newtheorem{conjecture}[theorem]{Conjecture}
\newtheorem{lemma}[theorem]{Lemma}
\newtheorem{proposition}[theorem]{Proposition}
\newtheorem{corollary}[theorem]{Corollary}
\newtheorem{remark}[theorem]{Remark}

\theoremstyle{definition} 
\newtheorem{definition}{Definition}[section]
\newtheorem{example}{Example}[section]

\title{On the Rasmussen--Tamagawa conjecture for abelian fivefolds}
\author[]{Shun Ishii}
\email{ishii.shun@keio.jp}
\address{{\footnotesize Department of Mathematics, Keio University, 3-14-1 Hiyoshi, Kouhoku-ku, Yokohama 223-8522, Japan.}}
\date{}

\begin{document}
\begin{abstract}
    In this paper, we study the Rasmussen--Tamagawa conjecture for abelian varieties with constrained prime power torsion. Previously, Rasmussen and Tamagawa have established the conjecture under the Generalized Riemann Hypothesis (GRH) for abelian varieties of any dimension over any number field, and unconditionally for those over $\mbQ$ of dimension at most $3$. We prove several cases of the conjecture by giving partial refinements of their techniques. Among other things, we give a proof of the conjecture for abelian fivefolds over $\mbQ$.
\end{abstract}

\maketitle

\tableofcontents

\section{Introduction}\label{sec:1}

\subsection{Rasmussen-Tamagawa conjecture}\label{subsec:1.1}

The aim of this paper is to prove several unconditional results about the Rasmussen--Tamagawa conjecture on abelian varieties with constrained torsion.

Throughout the section, let $K$ be a number field, $g$ a positive integer and $\ell$ a (rational) prime. For an abelian variety $A$ over $K$, we write $[A]$ for the $K$-isomorphism class of $A$, and define the set $\mcA(K, g, \ell)$ of isomorphism classes of certain $g$-dimensional abelian varieties over $K$ as follows:
\[
\mcA(K, g, \ell) \coloneqq \{ [A] \mid \text{$A$ has good reduction outside $\ell$ and } K(A[\ell^{\infty}]) \text{ is a pro-$\ell$-extension of $K(\zeta_{\ell})$} \}.
\]

Here $\zeta_{\ell}$ denotes a primitive $\ell$-th root of unity. Note that the first condition implies the finiteness of the set $\mcA(K, g, \ell)$ by the Shafarevich conjecture established by Faltings \cite{Fa83} and Zarhin \cite{Za85}. 

In \cite{RT08}, Rasmussen and Tamagawa have conjectured that $\mcA(K, g, \ell)$ is even empty if $\ell$ is suitably large enough. We refer the reader to \cite[\textsection 0]{RT08} for the original motivation for this problem, especially its connection with the pro-$\ell$ outer Galois representation associated to the projective line minus three points.

\begin{conjecture}[The Rasmussen--Tamagawa conjecture {\cite[Conjecture 1]{RT08}}]
    The set $\mcA(K, g, \ell)$ is empty for every sufficiently large $\ell$.
\end{conjecture}

In the rest of this paper, we abbreviate the statement of the Rasmussen-Tamagawa conjecture simply as RT($K,g$), namely:
\begin{center}
    RT($K,g$): The set $\mcA(K, g, \ell)$ is empty for every sufficiently large $\ell$.
\end{center} 

There are several preceding results about this conjecture. Most of them consider suitable subsets of $\mcA(K, g, \ell)$ and show they are empty for every large $\ell$. For example, Bourdon \cite{Bo15} ($g=1$) and Lombardo \cite{Lo18} (in general) consider a subset of $\mcA(K, g, \ell)$ consisting of CM abelian varieties, and Arai--Momose \cite{AM14} considers the case of QM abelian varieties. We note that certain analogues of the conjecture for  Galois representations are studied by Ozeki \cite{Oz11} and Ozeki--Taguchi \cite{OT14}, and are applied to the case of semistable abelian varieties (which is also established by Rasmussen--Tamagawa \cite[Theorem 3.6]{RT17}).

However, the conjecture in full generality largely remains open. Most notably, Rasmussen and Tamagawa obtained the following result:

\begin{theorem}[Rasmussen--Tamagawa {\cite[\textsection 7]{RT17}}]\label{thm:RT} \,
    \begin{enumerate}
        \item $\mathrm{RT}(\mbQ,g)$ holds for every $g \leq 3$.
        \item $\mathrm{RT}(K,1)$ holds if either
        \begin{itemize}
            \item $[K:\mbQ]=2$,
            \item $[K:\mbQ]=3$, or
            \item $K/\mbQ$ is a Galois extension whose Galois group has exponent $3$.
        \end{itemize}
        \item $\mathrm{RT}(K,g)$ holds for every $K$ and $g$, assuming {\rm GRH}.
    \end{enumerate}
\end{theorem}

There is also a uniform version of the conjecture: Let $n$ be a positive integer. The conjecture states that, if $\ell$ is sufficiently large, then $\mcA(K,g,\ell)=\emptyset$  for every $K$ with $[K:\mbQ]=n$ and $\ell>C(g,n)$.  The uniform version was also settled if $n$ is odd \cite[Theorem 5.2]{RT17} under GRH. We also refer the interested reader to recent results by Rasmussen-McLeman  \cite{MR24}\cite{MR25} for the case where $(g,n)=(1,2)$.

We roughly explain the strategy of Rasmussen and Tamagawa, assuming $K=\mbQ$ for simplicity: As Lemma \ref{lmm:mod-ell} shows, that $[A] \in \mcA(\mbQ,g,\ell)$ implies the semisimplification of $A[\ell]$ is isomorphic to  a direct sum of powers of the mod-$\ell$ cyclotomic character. This puts strong constraints on how a Frobenius element at $p \ll \ell$ acts on the $\ell$-adic Tate module of $A$. Based on this observation, they showed $p$ cannot be a $m_{\mbQ}$-th power residue modulo $\ell$, cf. Proposition \ref{prp:key} (here, $m_{\mbQ}$ is an integer attached to $A$). 

Hence the problem is reduced to finding a sufficiently small prime $m_{\mbQ}$-th power residue modulo $\ell$, and an effective version of the Chebotarev density theorem is  enough. In several cases, classical unconditional bounds are found to be enough: For example, Rasmussen and Tamagawa resort to a classical result of Elliott \cite[Theorem]{El71} on the least prime $k$-th power residue to show RT($\mbQ, g$) for $g \leq 3$.

They also use a certain classification of several numerical data attached to $[A] \in \mcA(K, g, \ell)$, cf. Proposition \ref{prp:key}. This plays a central role in proving the known cases of the conjecture. However, the classification becomes more complicated as the dimension increases, and even in the case of abelian fourfolds over $\mbQ$, there remains five such data, cf. Example \ref{ex:d=4}.

\subsection{Main result}\label{subsec:1.2}

The goal of this paper is to establish the Rasmussen--Tamagawa conjecture for abelian \emph{fivefolds} over $\mbQ$, which is somewhat unexpected in view of the preceding paragraph.

\begin{theorem}[=Theorem {\ref{thm:proof}}]\label{thm:main}
    $\mathrm{RT}(\mbQ,5)$ holds.
\end{theorem}

\begin{example}
    We give examples of primes $\ell$ such that $\mcA(\mbQ,5,\ell)$ is nonempty. First, note that we have an obvious map 
    \[
    \mcA(\mbQ, g, \ell) \times \mcA(\mbQ, g', \ell) \to \mcA(\mbQ, 5, \ell)
    \] for each pair of positive integers $(g,g')$ with $g+g'=5$. In particular, if $\mcA(\mbQ, 1, \ell)$ is nonempty, so is $\mcA(\mbQ, 5, \ell)$. According to the complete list of $\mcA(\mbQ, 1, \ell)$ \cite[Table 2]{RT08},  the set $\mcA(\mbQ,5,\ell)$ is nonempty at least if 
    \[
    \ell \in \{ 2,3,7,11,19,43,67,163 \}.
    \]
    A more nontrivial example comes from the Jacobian of the Fermat curve of degree $11$, which decomposes into absolutely simple abelian fivefolds $A_{1}, \dots, A_{9}$ having potential CM by $\mbQ(\mu_{11})$. They define pairwise distinct elements of $\mcA(\mbQ, 5, 11)$, though all of them are isomorphic over $\bar{\mbQ}$. It might be an interesting question to find a prime $\ell$ outside the above list such that $\mcA(\mbQ, 5, \ell)$ is nonempty.
\end{example}

The proof of Theorem \ref{thm:main} relies on Rasmussen--Tamagawa's strategy. Moreover, we introduce two minor but new insights to prove several particular cases of the conjecture: one is the use of small prime quadratic \emph{nonresidues}, and the other is the use of explicit prime quadratic residues. Such primes allow us to conclude that $[A] \in \mcA(\mbQ, g, \ell)$ with certain numerical data cannot exist, by considering the irreducible decompositions of the characteristic polynomials of Frobenius elements at these primes, cf. Propositions \ref{prp:mq8}-\ref{prp:mq2p}. Since we typically  assume that $g$ is odd in order to apply our arguments, this approach cannot be applied to the case of abelian fourfolds.

Regarding the first insight, the origin of the idea of using small quadratic nonresidues already can be found in \cite[Proposition 7.6]{RT17} which was crucial to prove RT($K,1$) where $K/\mbQ$ is a cubic extension. The second one does not seem to be discussed in any of the relevant papers.

Fortunately, these observations are enough to conclude $\mathrm{RT}(\mbQ,5)$. The two new ingredients also give nontrivial results on $\mathrm{RT}(K,2)$ and $\mathrm{RT}(K,3)$ where $K$ is a quadratic field, but we could not settle these cases in full generality at the writing of this paper. 

\medskip

{\bf Acknowledgements.} This work grew out of a learning seminar on the Rasmussen-Tamagawa conjecture held at Keio University. The author would like to express his deepest gratitude to \emph{Hidetada Wachi} and \emph{Naganori Yamaguchi} for helpful discussions and small talk at the seminar, both of which are essential for making the present paper possible. Throughout this work, the author was supported by JSPS KAKENHI Grant Number JP23KJ1882.

\medskip

{\bf Notation and Conventions.} Throughout the paper, we work with a fixed prime $\ell$, which is often assumed to be sufficiently large. 
\begin{itemize}
    \item We fix an algebraic closure $\bar{\mbQ}$ of $\mathbb{Q}$, and every number field is considered to be a subfield of $\bar{\mbQ}$. We also fix a primitive $n$-th root of unity $\zeta_{n} \in \bar{\mbZ}$ for each integer $n>0$. We write $\bar{\mbZ}$ for the integral closure of $\mbZ$ in $\bar{\mbQ}$. 
    
    \item The absolute Galois group $\gal(\bar{\mbQ}/K)$ of a number field $K$ is denoted by $G_{K}$. For a prime $\lambda$ of $K$ above $\ell$, we write $K_{\lambda}$  (resp. $\mathbb{F}_{\lambda}$) for the $\lambda$-adic completion of $K$ (resp. the residue field at $\lambda$). We denote the degree of $\mathbb{F}_{\lambda}/\mathbb{F}_{\ell}$ by $f_{\lambda/\ell}$ and the maximal unramified extension of $K_{\lambda}$ by $K_{\lambda}^{\mathrm{un}}$.
    \item We denote the mod-$\ell$ cyclotomic character by $\chi \colon G_{\mbQ} \to \mathbb{F}_{\ell}^{\times}$. 
    \item We fix a Frobenius element $\fr_{p} \in G_{\mbQ}$ at $p$ for every prime $p \neq \ell$. Note that $\chi(\fr_{p})=p \bmod \ell$.
\end{itemize}

\section{Review of Rasmussen--Tamagawa's work}\label{sec:2}

The aim of this section is to review  Rasmussen--Tamagawa's work \cite{RT17} for completeness. All results in this section are proved in op.cit, and no originality is claimed for this section.

\subsection{Constraints on numerical invariants}\label{sec:2.1}

First, we introduce the following simple but extremely useful lemma: 

\begin{lemma}[{\cite[Lemma 3.3]{RT17}}]\label{lmm:mod-ell}
    Let $A$ be a $g$-dimensional abelian variety over $K$, and 
    \[
    \rho_{A,\ell} \colon G_{K} \to \GL(A[\ell])
    \] the Galois representation attached to the $\ell$-torsion subgroup of $A$. Then $[A] \in \mcA(K,g,\ell)$ if and only if $A$ has good reduction everywhere outside $\ell$ and, after choosing a suitable basis of $A[\ell]$, the representation $\rho_{A,\ell}$ is of the form 
    \begin{align*}
        \rho_{A,\ell}=
        \begin{pmatrix}
            \chi^{i_{1}} & \ast & \cdots & \ast \\
             & \chi^{i_{2}} & \cdots & \ast \\
             &  & \ddots & \ast \\
             &  &  & \chi^{i_{2g}}
        \end{pmatrix}.
    \end{align*}
    Here, $i_{1}, \dots, i_{2g}$ are elements of $\mbZ/(\ell-1)\mbZ$.
\end{lemma}

 Rasmussen and Tamagawa associate $[A] \in \mcA(K,g,\ell)$ with several invariants, which are summarized as follows:

\begin{definition}\label{dfn:invariants}
    Suppose $[A] \in  \mcA(K,g,\ell)$.
    \begin{enumerate}
    \item We define $\{ i_{1}, i_{2}, \dots, i_{2g} \}$ to be the set of elements of $\mbZ/(\ell-1)\mbZ$ defined as in Lemma \ref{lmm:mod-ell}. This is well-defined modulo $|\chi(G_{K})|$, which is equal to $\ell-1$ for every $\ell \gg 0$.
    \item We define $m_{\mbQ}$ to be the order of the image of the homomorphism
    \[
    G_{\mbQ} \to (\mathbb{F}_{\ell}^{\times})^{2g} \,;\,  \sigma \mapsto (\chi^{2i_{r}-1}(\sigma))_{1 \leq r \leq 2g}.
    \] By definition, $m_{\mbQ}$ divides $\ell-1$.
    \item Suppose $\lambda$ is a place of $K$ lying above $\ell$, and let $\ell'$ be a prime different from $\ell$. We define $e_{\lambda}$ to be the order of the image of the $\ell'$-adic representation
    \[
        G_{K_{\lambda}^{\mathrm{un}}} \to \GL(V_{\ell'}(A_{K_{\lambda}})^{\mathrm{ss}})
    \] associated to the semisimplification of $V_{\ell'}(A_{K_{\lambda}})$. Then $e_{\lambda}$ is finite, and does not depend on the choice of $\ell'$ \cite[Expose IX, Theoreme 3.6 and \textsection 4.1]{SGA7}. When $K=\mbQ$, we abbreviate $e_{\ell}$ as $e$.
    \item We write 
    \[
    \chi(m_{\mbQ}) \colon G_{\mbQ} \to \mathbb{F}_{\ell}^{\times} \to \mathbb{F}_{\ell}^{\times}/(\mathbb{F}_{\ell}^{\times})^{m_{\mbQ}}
    \] for the compositum of the mod-$\ell$ cyclotomic character and the natural quotient map. 
    \end{enumerate}
\end{definition}

\begin{remark}\label{rmk:e} Suppose $\lambda$ is a place of $K$ above $\ell$ and $A$ is a $g$-dimensional abelian variety over $K_{\lambda}$. Let $e_{\lambda}$ be an integer defined in the same way as Definition \ref{dfn:invariants} (2).
    \begin{enumerate}
        \item Let $L/K_{\lambda}^{\mathrm{un}}$ be the Galois extension (of degree $e_{\lambda}$) corresponding to the kernel of the $\ell'$-adic representation
        \[
            G_{K_{\lambda}^{\mathrm{un}}} \to \GL(V_{\ell'}(A_{K_{\lambda}})^{\mathrm{ss}}). 
        \] Then $A_{L}$ has semistable reduction \cite[Expose IX, Corollaire 3.8]{SGA7}.
        \item Every prime divisor of $e_{\lambda}$ is less than or equal to $2g+1$  \cite[Corollary 3.2]{RT17}. As a corollary, $L/K_{\lambda}^{\mathrm{ur}}$ is a totally tamely ramified (and hence cyclic) extension whenever $\ell>2g+1$ \cite[Corollary 6.7]{RT17}.    
    \end{enumerate} 
\end{remark}

These invariants are strongly constrained, as the following lemma shows.

\begin{lemma}[{\cite[Lemma 3.5 and Lemma 4.1]{RT17}}]\label{lmm:e}
   Suppose $[A] \in \mcA(K,g,\ell)$. Then, for every sufficiently large $\ell$, the following assertions hold:
   \begin{enumerate}
    \item $\gcd(e, \ell-1)=\gcd(\frac{e}{2}, \ell-1)$.
    \item $4 \mid e$.
    \item for any $1 \leq r,s \leq 2g$, we have 
    \[
    \frac{e}{2}(i_{r}+i_{s}-1) \equiv 0 \bmod \ell-1.
    \] In particular, we have $m_{\mbQ} \mid \gcd(\frac{e}{2}, \ell-1)$.
    \item $\mathrm{ord}_{2}(m_{\mbQ})=\mathrm{ord}_{2}(\ell-1)$. Here $\mathrm{ord}_{2}$ stands for the normalized $2$-adic valuation.
   \end{enumerate}
\end{lemma}

This lemma enables them to show that traces of various powers of Frobenius elements at small primes take very specific forms. One of typical results is  as follows:

\begin{proposition}[{{\cite[Proposition 4.2]{RT17}}}]\label{prp:4.2}
    Suppose $\ell$ is sufficiently large and $[A] \in \mcA(\mbQ,g,\ell)$. For any prime $p<\frac{\ell}{4g}$, $\chi(m_{\mbQ})(\mathrm{Frob}_{p}) \neq 1$.
\end{proposition}

Combining this proposition with Elliott's result \cite[Theorem]{El71} on the least prime $k$-th power residue, they obtained the following result:

\begin{proposition}[{\cite[Proposition 4.4]{RT17}}]\label{prp:mq6}
    Suppose $\ell$ is sufficiently large. Then we have $m_{\mbQ}>6$ for every $[A] \in \mcA(\mbQ, g, \ell)$.
\end{proposition}

In the following, we give a slightly different proof for the case where $m_{\mbQ}=6$ and $g$ is odd. This argument is one of the prototypes of the results proved in the next section.

Suppose $\ell$ is sufficiently large, $[A] \in \mcA(\mbQ, g, \ell)$ and $m_{\mbQ}=6$. By Elliott's theorem \cite[Theorem]{El71}, there exists a prime quadratic residue $p$ modulo $\ell$ with $p^{3}<\ell/4g$. We have 
    \[
    \chi^{3(2i_{r}-1)}(\mathrm{Frob}_{p}) = 1 
    \quad (r=1, \dots, 2g).
    \]
    Let us consider the trace $a_{p,6}$ of $\mathrm{Frob}_{p}^{6}$ acting on $V_{\ell}(A)$. Then we have
    \begin{align*}
    a_{p,6} \equiv \sum_{r=1}^{2g} \chi^{6i_{r}}(\mathrm{Frob}_{p})
    = \left( \sum_{r=1}^{2g} \chi^{3(2i_{r}-1)}(\mathrm{Frob}_{p}) \right)p^{3} 
    = 2gp^{3} \bmod \ell.
    \end{align*}
    Since $|a_{p,6}| \leq 2gp^{3}$ by the Weil conjecture, it follows that $a_{p,6}=2gp^{3}$. 

    Let $\alpha_{p,1}, \dots, \alpha_{p,2g} \in \bar{\mathbb{Z}}$ be the roots of the characteristic polynomial of $\mathrm{Frob}_{p}$. That $a_{p,6}=2gp^{3}$ is equivalent to saying
    \[
    \alpha_{p,r}^{6}=p^{3} \quad (r=1, \dots, 2g),
    \] and one can observe that the irreducible decomposition of the polynomial $T^{6}-p^{3}$ is given by 
    \[
    T^{6}-p^{3}=(T^{2}-p)(T^{4}+pT+p^{2}).
    \] 
    Hence the characteristic polynomial of $\mathrm{Frob}_{p}$ is a product of $(T^{2}-p)^{2}$ and $T^{4}+pT+p^{2}$ since every real root of a $p$-Weil polynomial has even multiplicity. In particular, the degree of the characteristic polynomial must be a multiple of $4$, which is a contradiction since the degree is $2g$.

\subsection{Further constraints from special fibers}\label{sec:2.2}

In this subsection, we explain several results obtained in \cite[\textsection 5-6]{RT17} obtained via a detailed analysis of the action of the inertia subgroup on the special fibers of N\'eron models. First of all, we record the following result, which is stated before Lemma 3.5 of \cite{RT17}:

\begin{lemma}\label{lmm:goodred}
    Suppose $\ell$ is sufficiently large. If $[A] \in \mcA(K,g,\ell)$, then $A$ has potentially good reduction with $\ell$-rank zero at every place $\lambda$ of $K$ above $\ell$.
\end{lemma}

Suppose $\ell$ is sufficiently large, $[A] \in \mcA(K,g,\ell)$, and $\lambda$ is a place of $K$ above $\ell$. By Lemma \ref{lmm:goodred}, we may assume that $A$ has potentially good reduction at $\lambda$ with $\ell$-rank zero. 

Let $L/K_{\lambda}^{\mathrm{ur}}$ be the totally tamely ramified extension of degree $e_{\lambda}$. Then $A_{L}$ has good reduction by Remark \ref{rmk:e}. We denote its N\'eron model by $\mathcal{A}_{L}$, which is an abelian scheme over $O_{L}$. Since the natural action of the Galois group $\gal(L/K_{\lambda}^{\mathrm{ur}}) \cong \mbZ/e_{\lambda}\mbZ$ on $A_{L}$ uniquely extends to that on $\mathcal{A}_{L}$,  it induces the action on its special fiber denoted by $\bar{\mathcal{A}}_{L}$. Moreover, since the action of $\gal(L/K_{\lambda}^{\mathrm{ur}})$ on $\mathcal{A}_{L}$ is compatible with the natural action on $\mathrm{Spec}(O_{L})$, the action of $\gal(L/K_{\lambda}^{\mathrm{ur}})$ on $\bar{\mathcal{A}}_{L}$ preserves its $\bar{\mathbb{F}}_{\lambda}$-structure. 

Rasmussen and Tamagawa studied a decomposition of the special fiber $\bar{\mathcal{A}}_{L}$ associated to the action of 
\[
\mbZ[\gal(L/K_{\lambda}^{\mathrm{ur}})] \cong \mbZ[X]/(X^{e_{\lambda}}-1) \cong \prod_{d \mid e_{\lambda}} \mbZ[X]/(\Phi_{d}(X)),
\] where $\Phi_{d}(X)$ is the $d$-th cyclotomic polynomial. Let $\mathfrak{a}_{d}$ be the image of the ideal $ (\Phi_{d}(X)) \subset \mathbb{Z}[X]$ in $\mathrm{End}(\bar{\mathcal{A}}_{L})$, and define $B_{d} \coloneqq \bar{\mathcal{A}}_{L}/\mathfrak{a}_{d}\bar{\mathcal{A}}_{L}$, which is an abelian variety whose dimension is denoted by $g_{d}$.

\begin{proposition}[{\cite[Proposition 6.2 and Proposition 6.7 (C4)$\Rightarrow$ (C8)]{RT17}}]\label{prp:key} 
    Suppose  $\ell$ is sufficiently large, $[A] \in \mcA(K,g,\ell)$ and $\lambda$ is a place of $K$ above $\ell$. The following assertions hold.
    \begin{enumerate}
        \item Let $\varphi$ be Euler's totient function. There exist nonnegative integers $n_{d}$, indexed by the divisor $d$ of $e_{\lambda}$, such that
        \[
        2g=\sum_{d \mid e_{\lambda}} n_{d}\varphi(d)
        \quad \text{and} \quad e_{\lambda}=\mathrm{lcm}(d \mid n_{d}>0).
        \]
        \item We have
        \[
        n_{d}f_{\lambda}f_{\lambda/\ell} \neq 1.
        \] Here, $f_{\lambda}$ is the order of $\ell^{f_{\lambda/\ell}} \bmod d \in (\mbZ/d\mbZ)^{\times}$. In particular, if $K=\mbQ$, then $n_{d}=1$ implies $\ell \neq 1 \bmod d$.
    \end{enumerate}
\end{proposition}

Note that the condition (C4) in \cite[Proposition 6.7]{RT17} is that the $\ell$-rank of $B_{d}$ is zero, and is satisfied for every sufficiently large $\ell$ by Lemma \ref{lmm:goodred}. Proposition \ref{prp:key} allows us to prove the Rasmussen--Tamagawa conjecture in several cases, and some of typical applications are as follows:

\begin{example}[{\cite[Proposition 7.1]{RT17}}]
    Let us prove RT($\mbQ,2$) using Proposition \ref{prp:key}. Suppose $\ell \gg 0$ and $[A] \in \mcA(\mbQ, 2, \ell)$. By Proposition \ref{prp:key} (1), we have
    \[
    4=\sum_{d \mid e} n_{d}\varphi(d)
    \quad \text{and} \quad 
    e=\mathrm{lcm}(d \mid n_{d}>0)
    \]
    Moreover, we know that $4$ divides $e$ by Lemma \ref{lmm:e} (2). Since 
    \[
    \{d>0 \mid \varphi(d) \leq 4 \} = \{ 1,2,3,4,5,6,8,10,12  \},
    \]
    at least one of $n_{4}, n_{8}, n_{12}$ must be positive. Moreover, since $\varphi(8)=\varphi(12)=4$, we have
    \[
    4=\varphi(12) 
    \quad \left(\text{resp.} \quad
    4=\varphi(8) \right)
    \] if $n_{12}$ (resp. $n_{8}$) is positive. In both cases, it follows by Lemma \ref{lmm:e} that $m_{\mbQ} \leq 6$, contradicting Proposition \ref{prp:mq6}. If $n_{4}>0$, then we can easily check that $e=4$ or $e=12$, this also contradicts Proposition \ref{prp:mq6}. Hence $\mcA(\mbQ, 2, \ell)$ must be empty for every sufficiently large $\ell$. We note that a similar argument works for RT($\mbQ, 3$) \cite[Proposition 7.2]{RT17}.
\end{example}

\begin{example}[{\cite[Proposition 7.3]{RT17}}]\label{ex:d=4}
    The strategy of using Proposition \ref{prp:key} is successful when $K=\mbQ$ and $g \leq 3$, but it is not for abelian fourfolds. In fact, the conjecture remains open for the following five types of decompositions (and congruences), according to \cite[Proposition 7.3]{RT17}:
    \begin{center}
        \begin{tabular}{c|c}
            Sum & Congruence   \\
            \hline
            $2\varphi(3)+\varphi(8)$ & $\ell \equiv 13 \bmod 24$ \\
            $2\varphi(6)+\varphi(8)$ & $\ell \equiv 13 \bmod 24$ \\
            $\varphi(16)$ & $\ell \equiv 9 \bmod 16$ \\
            $\varphi(20)$ & $\ell \equiv 11 \bmod 20$ \\
            $\varphi(24)$ & $\ell \equiv 13 \bmod 24$ 
        \end{tabular}
    \end{center}
\end{example}

\section{Proof of main result}\label{sec:3}

\subsection{Using small quadratic nonresidues and explicit quadratic residues}\label{subsec:3.1}

We keep the same notation as in the previous section. Recall that, if $\ell$ is sufficiently large and $[A] \in \mcA(\mbQ, g, \ell)$, the integer $m_{\mbQ}$ cannot take $2$, $4$ and $6$. In what follows, we prove that it cannot take certain other values either. To obtain the result, we will use the following classical bound of the least quadratic nonresidue:

\begin{theorem}[Burgess {\cite[Theorem 2]{Bu57}}]\label{thm:nonres}
   For an odd prime $\ell$, let $n_{\ell}$ be the least odd quadratic nonresidue modulo $\ell$, i.e. the least positive integer which is not a quadratic residue modulo $\ell$. Then, for any $\epsilon>0$, we have
   \[
   n_{\ell}=O(\ell^{\frac{1}{4\sqrt{e}}+\epsilon})
   \] where $e=2.718\dots$ is Euler's number.
\end{theorem}

Note that the least odd quadratic nonresidue modulo $\ell$ is necessarily a prime. Although \cite[Theorem 2]{Bu57} only gives the upper bound of the least quadratic nonresidue, the same proof works for the the least odd quadratic nonresidue.

\begin{proposition}\label{prp:mq8}
     Suppose $\mathrm{ord}_{2}(g) \leq 1$ and $\ell$ is sufficiently large. Then we have $m_{\mbQ} \neq 8$ for every $[A] \in  \mcA(\mbQ, g, \ell)$.
\end{proposition}
\begin{proof}
    Assume that $[A] \in  \mcA(\mbQ, g, \ell)$ and $m_{\mbQ}=8$. Note that we have $8 \mid m_{\mathbb{Q}} \mid \ell-1$. By Theorem \ref{thm:nonres}, if $\ell$ is sufficiently large, we can always find a prime $p$ that satisfies the following two conditions:
    \begin{enumerate}
        \item $p^{4}<\ell/4g$, and 
        \item $p$ is a quadratic nonresidue modulo $\ell$.
    \end{enumerate}

    We consider the action of $\fr_{p}^{8}$ on $V_{\ell}(A)$. We denote its trace by $a_{p,8}$ and the roots of the characteristic polynomial by $\alpha_{p,1}, \dots \alpha_{p,2g} \in \bar{\mbZ}$. By the Weil conjecture, we have 
    \[
    |\alpha_{p,r}|=\sqrt{p}
    \quad \text{for $r=1, \dots, 2g$}.
    \] 

    That $p$ is a quadratic nonresidue modulo $\ell$ is equivalent to saying $p^{\frac{\ell-1}{2}} \equiv -1 \bmod \ell$. Since $2i_{r}-1$ is odd, it follows that
    \[
    \chi^{4(2i_{r}-1)}(\fr_{p}) \equiv -1 \bmod \ell 
    \quad \text{for $r=1, \dots, 2g$}.
    \]    
    Hence we have
    \begin{align*}
        a_{p,8} = \sum_{r=1}^{2g} \alpha_{p,r}^{8} \equiv \sum_{r=1}^{2g} \chi^{8i_{r}}(\fr_{p})   
        =  \left( \sum_{r=1}^{2g} \chi^{4(2i_{r}-1)}(\fr_{p}) \right) p^{4} = -2gp^{4} \bmod \ell.       
    \end{align*}

    Since $
        |a_{p,8}| \leq \sum_{r=1}^{2g} |\alpha_{p,r}|^{8}=2gp^{4},
    $ we have $a_{p,8}=-2gp^{4}$ by the first condition on $p$. This equality occurs if and only if $\alpha_{p,r}^{8}=-p^{4}$ for every $r=1, \dots, 2g$. In particular, every root of the characteristic polynomial of $\fr_{p}$ is also a root of $T^{8}+p^{4}$.
    
    Now $T^{8}+p^{4}$ is irreducible over $\mbQ$. This follows, for example, by considering the action of the Galois group $\gal(\mbQ(\zeta_{16}, \sqrt{p})/\mbQ)$ on the set of roots of $T^{8}+p^{4}$. The irreducibility also follows from \cite[Chapter VI, Theorem 9.1]{La02}. Hence the degree of the characteristic polynomial of the Frobenius, which is $2g$, must be divisible by $8$, contradicting the assumption that $\mathrm{ord}_{2}(g) \leq 1$.
\end{proof}

Note that $6<4\sqrt{e}=6.594\dots<7$. That $6<4\sqrt{e}$ is essential to prove the following result, whose proof is similar to the last proposition:

\begin{proposition}\label{prp:mq12}
    Suppose $g$ is odd and $\ell$ is sufficiently large. Then we have $m_{\mbQ} \neq 12$ for every $[A] \in  \mcA(\mbQ, g, \ell)$.
\end{proposition}
\begin{proof}
    By Theorem \ref{thm:nonres}, if $\ell$ is sufficiently large, we can always find an \emph{odd} prime $p$ that satisfies the following two conditions:
    \begin{enumerate}
        \item $p^{6}<\ell/4g$, and 
        \item $p$ is a quadratic nonresidue modulo $\ell$.
    \end{enumerate}

    Consider the trace $a_{p,12}$ of $\fr_{p}^{12}$ acting on $V_{\ell}(A)$ and denote the roots of the characteristic polynomial by $\alpha_{p,1}, \dots \alpha_{p,2g}$. We have
    \[
    \chi^{6(2i_{r}-1)}(\fr_{p}) \equiv -1 \bmod \ell 
    \quad \text{for $i=1, \dots, 2g$},
    \] and hence 
    \begin{align*}
        a_{p,12} &= \sum_{r=1}^{2g} \alpha_{p,r}^{12} \equiv \sum_{r=1}^{2g} \chi^{12i_{r}}(\fr_{p})  \bmod \ell \\
        &=  \left( \sum_{r=1}^{2g} \chi^{6(2i_{r}-1)}(\fr_{p}) \right) p^{6} = -2gp^{6} \bmod \ell.     
    \end{align*}
    Since $|a_{p,12}| \leq \sum_{r=1}^{2g} |\alpha_{p,r}|^{12}=2gp^{6}$ and $p^{6}<\frac{\ell}{4g}$, it follows that $a_{p,12}=-2gp^{6}$. Moreover, this equality occurs if and only if $\alpha_{p,r}^{12}=-p^{6}$ for every $r=1, \dots, 2g$. 
    
    Now we claim that
    \[
    T^{12}+p^{6}=(T^{4}+p^{2})(T^{8}-p^{2}T^{2}+p^{4})
    \] gives the irreducible decomposition of $T^{12}+p^{6}$ over $\mbQ$\footnote{we use $p \neq 2$ here since $T^{4}+2^{2}=(T^{2}+2T+2)(T^{2}-2T+2)$ is not irreducible.}. The irreducibility of $T^{4}+p^{2}$ follows from \cite[Chapter VI, Theorem 9.1]{La02}. Regarding that of $T^{8}-p^{2}T^{2}+p^{4}$, consider the action of $\gal(\mbQ(\zeta_{24}, \sqrt{p})/\mbQ)$ on the set 
    \[
    \{ \zeta_{8}^{2m+1} \zeta_{3}^{\pm 1} \sqrt{p} \mid m=0,1,2,3 \}
    \] of roots of $T^{8}-p^{2}T^{2}+p^{4}$. If $p \neq 3$, then  $\mbQ(\zeta_{8})$, $\mbQ(\zeta_{3})$ and $\mbQ(\sqrt{p})$ are linearly disjoint, and this immediately implies the transitivity of the action, and hence the desired irreducibility follows. When $p=3$, it is easy to check that $T^{8}-3^{2}T^{2}+3^{4}$ is irreducible. Hence the degree of the characteristic polynomial must be divisible by $4$, which is a contradiction.
\end{proof}

Lastly, we use explicit quadratic residues to exclude some possible values of $m_{\mbQ}$:

\begin{proposition}\label{prp:mq2p}
    Suppose $g$ is odd, $p$ is a prime congruent to $1$ modulo $4$ and $\ell$ is sufficiently large. Then we have $m_{\mbQ} \neq 2p$ for every $[A] \in  \mcA(\mbQ, g, \ell)$.
\end{proposition}
\begin{proof}
    We may assume that $\ell>\frac{p^{p}}{4g}$. Suppose $[A] \in  \mcA(\mbQ, g, \ell)$ and $m_{\mbQ}=2p$.  Then $\ell$ is congruent to $1$ modulo $p$, and hence $p$ is a quadratic residue modulo $\ell$ by the quadratic reciprocity law.

    It follows that 
    \[
    \chi^{p(2i_{r}-1)}(\fr_{p}) \equiv 1 \bmod \ell 
    \quad \text{for $i=1, \dots, 2g$}.
    \]

    Let $a_{p,2p}$ be the trace of $\fr_{p}^{2p}$ acting on $V_{\ell}(A)$ and $\alpha_{p,1}, \dots \alpha_{p,2g} \in \bar{\mbZ}$ roots of the characteristic polynomial. Then we have
    \begin{align*}
        a_{p,2p} = \sum_{r=1}^{2g} \alpha_{p,r}^{2p} \equiv \sum_{r=1}^{2g} \chi^{2pi_{r}}(\fr_{p})  \bmod \ell \\
        =  \left( \sum_{r=1}^{2g} \chi^{p(2i_{r}-1)}(\fr_{p}) \right) p^{p} = 2gp^{p} \bmod \ell.     
    \end{align*}
      Since $|a_{p,2p}| \leq 2gp^{p}$, it follows that $a_{p,2p}=2gp^{p}$. Hence $\alpha_{p,1}, \dots \alpha_{p,2g}$ must be roots of $T^{2p}-p^{p}$. Since the action of $\gal(\mbQ(\zeta_{p})/\mbQ)$ on the set \[
        \{ \pm \zeta_{p}^{i}\sqrt{p} \mid i=1,\dots, p-1 \}
        \] of roots of $\frac{T^{2p}-p^{p}}{T^{2}-p}$ decomposes it into two orbits
        \[
        \{ \left( \tfrac{a}{p} \right)  \zeta_{p}^{a}\sqrt{p} \mid i=1,\dots, p-1 \} 
        \quad \text{and} \quad
        \{ -\left( \tfrac{a}{p} \right)  \zeta_{p}^{a}\sqrt{p} \mid i=1,\dots, p-1 \},
        \] the polynomial $T^{2p}-p^{p}$ is a product of $T^{2}-p$ and two irreducible polynomials of degree $p-1$. Since every real root of a $p$-Weil polynomial has even multiplicity, the characteristic polynomial must be a product of $(T^{2}-p)^{2}$ and irreducible polynomials of degree $p-1$. This is a contradiction.
\end{proof}

    The results in this subsection suggest that irreducible decompositions of $p$-Weil \emph{binomials} are often useful to solve particular cases of the Rasmussen--Tamagawa conjecture. There is room for improvement, but the present form is enough to derive RT($\mbQ, 5$).

As a corollary, we prove the following partial result, which partially generalizes the Rasmussen--Tamagawa conjecture for semistable abelian varieties \cite[Theorem 3.6]{RT17}:

\begin{corollary}
        Suppose $g$ is odd, and let $\mcA_{24}(\mbQ,g,\ell)$ be a subset of $\mcA(\mbQ,g,\ell)$ consisting of $[A] \in \mcA(\mbQ,g,\ell)$ with $e \mid 24$. Then $\mcA_{24}(\mbQ,g,\ell)$ is empty for every sufficiently large $\ell$.
\end{corollary}

\begin{proof}
    The claim follows immediately from Propositions \ref{prp:mq6}, \ref{prp:mq8}, \ref{prp:mq12} and \ref{prp:mq2p}.
\end{proof}

\subsection{Proof of the Rasmussen--Tamagawa conjecture for abelian fivefolds}\label{subsec:3.2}

The aim of this subsection is to provide a proof of Theorem \ref{thm:main}.

\begin{theorem}[=Theorem \ref{thm:main}]\label{thm:proof}
    $\mathrm{RT}(\mbQ, 5)$ holds.
\end{theorem}

\begin{proof}
    Suppose $\ell \gg 0$ and $[A] \in \mcA(\mbQ, 5, \ell)$. By Proposition \ref{prp:key}, we have
    \[
        10=\sum_{d \mid e} n_{d}\varphi(d)
        \quad \text{and} \quad 
        e=\mathrm{lcm}(d \mid n_{d}>0).
    \]
    Recall that every prime divisor of $e$ is less than or equal to $2g+1=11$. If $11$ divides $e$, then we must have
    \[
    10=\varphi(11) 
    \quad \text{and} \quad 
    e=11, 
    \quad \text{or} \quad 
    10=\varphi(22) 
    \quad \text{and} \quad 
    e=22. 
    \] Both cases contradict Lemma \ref{lmm:e} (2). Therefore, every prime divisor of $e$ is either $2,3,5$ or $7$. Since $m_{\mbQ}$ divides $\gcd(\frac{e}{2},\ell-1)$ by Lemma \ref{lmm:e} (3), every prime divisor of $m_{\mbQ}$ is also equal to $2,3,5$ or $7$.

    First, let us suppose that $m_{\mbQ}$ is divisible by $7$. Since $m_{\mbQ}$ is even and is a divisor of $\ell-1$, it follows that
    \[
    \ell \equiv 1 \bmod 14.
    \] On the other hand, in this case, one of $n_{7}$ and $n_{14}$ must be one since $7$ also divides $e$. It follows by Proposition \ref{prp:key} (2) that $\ell$ is not congruent to $1$ modulo $14$. This is a contradiction.

    Secondly, suppose that $m_{\mbQ}$ is not divisible by $5$ or $7$. In other words, every prime divisor of $m_{\mbQ}$ is either $2$ or $3$. In this case, we claim that $9$ does not divide $m_{\mbQ}$. In fact, suppose that $9$ divides $m_{\mbQ}$ and hence $\ell \equiv 1 \bmod 18$. It is easy to confirm that
    \[
    \{ n>0 \mid \varphi(9n) \leq 10  \} = \{ 1, 2 \}
    \quad \text{and} \quad
    \varphi(9)=\varphi(18)=6.
    \] It follows that either $n_{9}$ or $n_{18}$ must be one, and this implies  $\ell \not \equiv 1 \bmod 18$ by Proposition \ref{prp:key}. This is a contradiction.

    Similarly, it is easy to show that $16$ does not divide $m_{\mbQ}$. Therefore, it follows that
    \[
    m_{\mbQ} \in \{ 2,4,6,8,12,24 \}.
    \] 
    By Proposition \ref{prp:mq6}, Proposition \ref{prp:mq8} and Proposition \ref{prp:mq12}, it follows that $m_{\mbQ}$ must be $24$. In this case, since $16$ divides $e$, it follows that $n_{16}=1$. Hence we have
    \[
    2=\sum_{d \mid e} n_{d}\varphi(d) - \varphi(16) 
    \quad \text{and} \quad 
    e=\mathrm{lcm}(d \mid n_{d}>0).
    \] Moreover, since $3$ also divides $e$, either $n_{3}$ or $n_{6}$ must be one. By Proposition \ref{prp:key} (2), this implies that $\ell \not \equiv 1 \bmod 6$, which contradicts $24=m_{\mbQ} \mid \ell-1$.

    Lastly, suppose $5$ divides $m_{\mbQ}$ (and $7$ does not). Note that, since 
    \[
            \{ n>0 \mid \varphi(5n) \leq 10 \}
            =
            \{ 1,2,3,4,6 \},
    \] at least one of $n_{5}, n_{10}, n_{15}, n_{20}, n_{30}$  is positive. If $n_{30}>0$, then it must be one since $\varphi(30)=8$. Moreover, since
    \[
        \{ n>0 \mid \varphi(n) \leq 2 \} = \{ 1,2,3,4,6 \},
    \] it follows that 
    $
    e \in \{ 30,60 \}
    $ and hence $m_{\mbQ} \leq 30$. We can apply the same argument to the case where $n_{15}$ or $n_{20}$ is positive to conclude $m_{\mbQ} \leq 30$.

    Now let us assume that $n_{5}$ or $n_{10}$ is positive. By Proposition \ref{prp:key} (2), it must be greater than one, since otherwise $10$ would not divide $\ell-1$. Since $2\varphi(5)=2\varphi(10)=8$, one can easily confirm that $m_{\mbQ} \leq 30$ also in this case.  To summarize, it holds that $m_{\mbQ} \leq 30$. In what follows, we derive a contradiction from this inequality.

    \begin{enumerate}
        \item $m_{\mbQ}=10$ : This contradicts Proposition \ref{prp:mq2p} for $p=5$.
        \item $m_{\mbQ}=20$ : In this case, $20 \mid \ell-1$ and $40 \mid e$. Moreover, it is easy to confirm that either $n_{5}$ or $n_{10}$ must be one to realize $m_{\mbQ}=20$, implying $10 \nmid \ell-1$ by Proposition \ref{prp:key} (2). This is a contradiction.
        \item $m_{\mbQ}=30$ : Note that $30 \mid \ell-1$.
        \begin{enumerate}
            \item If $n_{30}=1$, $30 \nmid \ell-1$ by Proposition \ref{prp:key} (2). 
            \item  Suppose $n_{20}=1$. Since $\varphi(20)=8$ and $3$ divides $e$, either $n_{3}$ or $n_{6}$ is also equal to one. Hence $6 \nmid \ell-1$ by Proposition \ref{prp:key} (2).
            \item If $n_{15}=1$, $15 \nmid \ell-1$ by Proposition \ref{prp:key} (2). 
            \item If $n_{10}>0$, then it must be greater than one by Proposition \ref{prp:key} (2). However, since $3$ divides $e$, either $n_{3}$ or $n_{6}$ must be one. Hence $6 \nmid \ell-1$. The same argument works if $n_{5}>0$.
        \end{enumerate}
    \end{enumerate}
     This completes the proof.
\end{proof}

Finally, let us briefly discuss RT($\mbQ, 7$). Suppose $\ell \gg 0$ and $[A] \in \mcA(\mbQ,7,\ell)$. Using Rasmussen-Tamagawa's techniques explained in the last section, one can check that\footnote{We thank Naganori Yamaguchi for kindly providing us a computer program to calculate possible values of $m_{\mbQ}$. }
\[
m_{\mbQ} \in \{ 8,10, 12, 14,18,20,24,30 \}.
\]
By Proposition \ref{prp:mq8} and Proposition \ref{prp:mq12}, we have $m_{\mbQ} \neq 8, 12$. Proposition \ref{prp:mq2p} shows that $m_{\mbQ} \neq 10$, and a similar argument proves that $m_{\mbQ}$ cannot be $30$ (consider the factorization of $T^{30}-5^{15}$). There are four possibilities left, namely:
\[
m_{\mbQ} \in \{ 14,18,20,24 \}.
\]
In any case, our result is not enough. For example, suppose $m_{\mbQ}=14$. Then $7$ is a quadratic nonresidue modulo $\ell$ since $\ell \equiv 3 \bmod 4$. Then one can show that every root of the characteristic polynomial of $\fr_{7}$ is a root of $T^{14}+7^{7}$, which factors as
\[
(T^2 + 7) (T^6 - 7T^5 + 21T^4 - 49T^3 + 147T^2 - 343T + 343)  (T^6 + 7T^5 + 21T^4 + 49T^3 + 147T^2 + 343T + 343).
\] This does not lead to a contradiction, since, for example, $T^{14}+7^{7}$ itself can be the characteristic polynomial of $\mathrm{Frob}_{7}$. On the other hand, if there is a prime quadratic residue $p$ modulo $\ell$ with $p^{7}<\ell/28$, then considering the factorization of $T^{14}-p^{7}$ yields $m_{\mbQ} \neq 14$ , cf. the discussion at the end of Section \ref{sec:2.1}.

As another example, suppose $m_{\mbQ}=24$. A similar argument to Proposition \ref{prp:mq12} works if there exists a prime quadratic nonresidue $p$ modulo $\ell$ such that $p^{12}<\ell/28$, and the present bound (Theorem \ref{thm:nonres}) is not enough to find such a prime. 

\bibliographystyle{amsalpha}
\bibliography{references}

\end{document}